\documentclass[12pt]{amsart}
\usepackage{amssymb,amsmath,amscd,graphicx,
latexsym,amsthm, amscd}
\usepackage{amssymb,latexsym,eufrak,amsmath,amscd,graphicx}
  \usepackage[all]{xy}
  \usepackage{pdfsync}
\usepackage{graphicx,mathrsfs,setspace,fancyhdr}
\theoremstyle{plain}
\newtheorem{theorem}{Theorem}

\newtheorem{lemma}[theorem]{Lemma}

\newtheorem{proposition.definition}[theorem]{Proposition/Definition}

\newtheorem{theoremalpha}{Theorem}
\newtheorem{corollaryalpha}[theoremalpha]{Corollary}
\newtheorem{propositionalpha}[theoremalpha]{Proposition}

\theoremstyle{definition}

\newtheorem{remark}[theorem]{Remark}

\newtheorem{problem}[theorem]{Problem}
\newtheorem{conjecture}[theorem]{Conjecture}

\newcommand{\lra}{\longrightarrow}

\newcommand{\noi}{\noindent}
\newcommand{\PP}{\mathbf{P}}
\newcommand{\PPP}{\PP_{\text{sub}}}

\newcommand{\ZZ}{\mathbf{Z}}
\newcommand{\CC}{\mathbf{C}}

\newcommand{\OO}{\mathcal{O}}

\newcommand{\frakm}{\mathfrak{m}}

\newcommand{\rk} {\text{rank }}

\newcommand{\HH}[3]{H^{{#1}} \big( {#2} , {#3}
\big) }

\newcommand{\lin}{\equiv_{\text{lin}}}

\newcommand{\Linser}[1]{| \mspace{1.5mu} {#1}
\mspace{1.5mu} |}
\newcommand{\linser}[1]{\Linser{  {#1}  }}

\newcommand{\ol}[1]{\overline{#1}}

\newcommand{\Pic}{\textnormal{Pic}}

\numberwithin{theorem}{section}
\numberwithin{equation}{section}

\begin{document}

\title{Stability of syzygy bundles on an algebraic surface}

 \author{Lawrence Ein}
  \address{Department of Mathematics, University Illinois at Chicago, 851 South Morgan St., Chicago, IL  60607}
 \email{{\tt ein@uic.edu}}
 \thanks{Research of the first author partially supported by NSF grant DMS-1001336.}

 \author{Robert Lazarsfeld}
  \address{Department of Mathematics, University of Michigan, Ann Arbor, MI
   48109}
 \email{{\tt rlaz@umich.edu}}
 \thanks{Research of the second author partially supported by NSF grant DMS-1159553.}
 
 \author{Yusuf Mustopa}
 \address{Department of Mathematics, Boston College, Chestnut Hill, MA 02467}
 \email{{\tt mustopa@bc.edu}}
 \thanks{Research of the third author partially supported by NSF grant RTG DMS-0502170.}
 
\maketitle

\section*{Introduction}
 
The purpose of this note is to prove the stability of the syzygy bundle associated to any sufficiently positive embedding of an algebraic surface.

 Let $X$ be a smooth projective algebraic variety over an algebraically closed field $k$, and let $L$ be a very ample line bundle on $X$. The \textit{syzygy} (or \textit{kernel}) \textit{bundle} $M_L$ associated to $L$ is by definition the kernel of the evaluation map 
 \[  \text{eval}_L : H^0(L) \otimes_k\OO_X \lra L. \]
 Thus $M_L$ sits in an exact sequence
 \[   0 \lra M_L \lra H^0(L) \otimes_k\OO_X \lra L \lra 0. \]
 These vector bundles (and some analogues) arise in a variety of geometric and algebraic problems, ranging from  the syzygies of $X$ to questions of tight closure.  Consequently  there has been considerable interest in trying to establish the stability of $M_L$ in various settings. When $X$ is a smooth curve of genus $g\ge 1$, the situation is well-understood thanks to the work of several authors (\cite{PR}, \cite{Bu}, \cite{EL1}, \cite{Bea}, \cite{Cam1}, \cite{Mis}):  in particular $M_L$ is stable as soon as $\deg L \ge 2g +1$. When $X = \PP^n$ and $L = \OO_{\PP^n}(d),$ the stability of $M_L$ was established by Flenner \cite[Cor. 2.2]{Fl} in characteristic $0$ and by Trivedi \cite{Tri} in characteristic $> 0$ for many $d$. A more general statement, due to Coand\v{a} \cite{Coa}, treats the bundles associated to possibly incomplete linear subseries of $\HH{0}{\PP^n}{\OO_{\PP^n}(d)}$.  Motivated by questions of tight closure, the  stability of syzygy bundles on $\mathbf{P}^{n}$ arising from a somewhat more general construction has been studied by Brenner \cite{Bre} and by Costa, Marques and Mir\'{o}-Roig \cite{CMMR}, \cite{MMR}.  In dimension $2$, Camere \cite{Cam2} recently proved that  kernel bundles on $K3$ and abelian surfaces are stable. 
 
We show here that if $L$ is a sufficiently positive divisor on any smooth projective surface $X$, then $M_L$ is stable with respect to a suitable hyperplane section of $X$.  Specifically, fix an ample divisor $A$ and an arbitrary divisor $P$ on $X$. Given a large integer $d$, set
 \[   L_d \ = \ dA + P, \]
 and write $M_d = M_{L_d}$.
Our main result is
\begin{theoremalpha} \label{Theorem.A.Intro}
If $d$ is sufficiently large, then $M_d$ is slope stable with respect to $A$. 
\end{theoremalpha}
\noi Recall that the conclusion means that if $F \subseteq M_d$ is a subsheaf with $0 < \rk(F) < \rk(M_d)$,  then 
\[   \frac{ c_1(F) \cdot A}{\rk F} \ < \ \frac{c_1(M_d) \cdot A }{\rk {M_d}}. \] 

Since a slope-stable bundle is also Gieseker stable, it follows that $M_d$ is parametrized for $d \gg 0$ by a point on the moduli space of bundles on $X$ with suitable numerical invariants. On the other hand, 
working over $\CC$, Camere \cite[Proposition 2]{Cam2} shows that if $H^1(X, \OO_X) = 0$, and if the natural map
\[  H^0(X, K_X) \otimes H^0(X, L) \lra H^0(X, K_X + L)   \]
is surjective for some very ample line bundle $L$, then $M_L$ is rigid. But this surjectivity is automatic if  $K_X$ is globally generated and   $L$ is sufficiently positive. Hence we deduce
\begin{corollaryalpha}
Let $X$ be a complex projective surface with vanishing irregularity $q(X) = 0$, and assume that $K_X$ is globally generated. Then   $M_d$ corresponds to an isolated point of the moduli space of stable vector bundles on $X$ when $d \gg 0$. \end{corollaryalpha}

It is natural to suppose that the analogue of Theorem \ref{Theorem.A.Intro} holds also for varieties of dimension $\ge 3$, but unfortunately our proof does not work in this setting. However if $\Pic(X) \cong \ZZ$, then the argument of Coand\v{a} \cite{Coa} goes through with little change to establish:
\begin{propositionalpha} \label{PropB.Intro} Assume that $\dim X \geq 2$ and that $\Pic(X) = \ZZ \cdot [A]$ for some ample divisor $A$. Write $L_d = dA$. Then $M_d =_{\textnormal{def}} M_{L_d}$ is $A$-stable for $ d \gg 0$. 
\end{propositionalpha}
As in \cite{Cam2} the strategy for Theorem \ref{Theorem.A.Intro} is to reduce the question to the stability of syzygy bundles on curves, but we avoid the detailed calculations appearing in that paper. In order to explain the idea, we sketch a quick proof of Camere's result \cite[Theorem 1]{Cam2} that if $L$ is a globally generated ample line bundle on a K3 surface $X$, then $M_L$ is $L$-stable. Supposing to the contrary, let $F \subseteq M_L$ be a saturated destabilizing 
subsheaf, and fix a general point $x \in X$. Consider now a general curve $C \in |L \otimes \frakm_{x}|$; we may suppose that $F$ sits as a sub-bundle of $M_L$ along $C$.  Restriction to  $C$ yields a diagram:
\begin{equation}\label{Intro.Diagram}
\begin{gathered}
\xymatrix{
&  & F |{C}    \ar@{^{(}->}[d] \\ 0 \ar[r] & \OO_C \ar[r]& M_L|{C} \ar[r] & \ol{M}_{\Omega_C} \ar[r]&0,
}
\end{gathered}\tag{*}
\end{equation}
where $\ol{M}_{\Omega_C}$ is the syzygy bundle on $C$ associated to $\Omega_C = L|{C}$. But $\ol{M}_{\Omega_C}$ is semi-stable by \cite{PR}, while 
\[  \mu \big( F|{C} \big) \ \ge \ \mu \big( M_L|{C}) \ > \ \mu \big (\ol{M}_{\Omega_C}\big). \]
It follows that $F|{C}$ cannot inject into $\ol{M}_{\Omega_C}$, and hence the two sub-bundles $F|_{C}$ and $\OO_C$ of $M_L|{C}$ have a non-trivial intersection, which   in turn implies that $\OO_C$ is contained in  $F|_{C}$.
 On the other hand, consider  the fibres at $x$ of the various bundles in play. The vertical map in (*) corresponds to a fixed subspace $F(x) \subsetneqq H^0(X, L \otimes \frakm_x)$. So we would be asserting that the equation defining a general curve $C \in  \linser{L \otimes \frakm_x}$ lies in this subspace, and this is certainly not the case. The proof of Theorem \ref{Theorem.A.Intro} in general proceeds in an analogous manner, the main complication being that  we have to deal with a trivial vector bundle of higher rank appearing in the bottom row of (*). 

Concerning the organization of the paper, Section 1 is devoted to the proof of  Theorem \ref{Theorem.A.Intro}. Proposition \ref{PropB.Intro} appears in Section 2, where we also propose some open problems. 

We are grateful to M. Musta{\c t}{\u a} and V. Srinivas for valuable discussions.

\section{Proof of Main Theorem}

This section is devoted to the proof of Theorem A from the Introduction. 

We start by fixing notation and set-up. As in the Introduction,  $X$ is a smooth projective surface, and $L_d = dA + P$ where $A$ is an ample divisor,  and $P$ is an arbitrary divisor on $X$. For the duration of the argument we fix an integer $b \gg 0 $ such that $(bA + P - K_X)$ is very ample, and so that $H^1(X, L_b) = 0$, and put 
\[  B \ =\ L_b \ =  \ bA +P.\] Observe that $b$ and $B$ are independent of $d$. We also assume henceforth that $d$
 is sufficiently large so that $L_d$ and $L_{d-b}$ are  very ample.

Fixing  such an  integer $d$, assume now that $ M_d = M_{L_d}$ is not $A$-stable. Recall that this means that there exists a non-trivial subsheaf 
\[     F_d \ \subseteq \ M_{d} \]
such that 
\[   \frac{ c_1(F_d) \cdot A}{\rk F_d} \ \ge \ \frac{c_1(M_d) \cdot A }{\rk {M_d}}. \] 
Without loss of generality we assume that $F_d \subseteq M_d$ is saturated, and we fix a point $x = x_d \in X$ at which $F_d$ is locally free.   Writing $\frakm_x$ for the ideal sheaf of $x$, we suppose finally that $d$ is always large enough so that the natural mapping
  \begin{equation} \label{Mult.C.Surjective}
 \HH{0}{X}{(d-b)A \otimes \frakm_x} \otimes  \HH{0}{X}{B} \lra \HH{0}{X}{L_d \otimes \frakm_x}
  \end{equation}
is surjective. 

The plan is to use the stability of syzygy bundles on curves to show that if $d \gg 0$, then no such $F_d$ can actually exist. To this end, consider a general curve  \[    C_d   \ \in \ \linser{(d-b)A} \, = \, \linser{L_d - B} \]
passing through the fixed point $x\in X$. We may assume that  $C_d$ is smooth and irreducible, and  that $M_d/F_d$ is locally free along $C_d$. Observe also that for any torsion free sheaf $\mathcal{F}$ on $X$ that is locally free along $C_d$, one has
\[   \mu_A\big(\mathcal{F}\big) \ = \ \frac{1}{(d-b)} \cdot \mu \big( \mathcal{F}|C_d \big). \]
In particular, if $\mathcal{F}$ is $A$-unstable as a sheaf on $X$, then $\mathcal{F}|C_d$ is unstable on $C_d$.
  
We now consider the restriction of $M_d$ and $F_d$ to $C_d$. Writing $\ol{M}_d = M_{{\ol{L}_d}}$ for the syzygy bundle on $C_d$ of the restriction $\ol{L}_d=L_d|C_d$, we have an exact sequence
\[  0 \lra H^0(B)_{C_d} \lra M_d|C_d \lra \ol{M}_d \lra 0, \]
where the term on the left is the trivial bundle on $C_d$ with fibre $H^0(X,B)$. We  complete this to a diagram
\begin{equation}\label{Basic.Diagram}
\begin{gathered}
\xymatrix{
0\ar[r] & \ol{K}_d \ar[r] \ar@{^{(}->}[d]& F_d\mid C_d \ar[r] \ar@{^{(}->}[d]& \ol{N}_d \ar[r] \ar@{^{(}->}[d]& 0 \\ 0 \ar[r] & H^0(B)_{C_d} \ar[r]& M_d|C_d \ar[r] & \ol{M}_d \ar[r]&0
}
\end{gathered}
\end{equation}
of vector bundles on $C_d$, where $\ol{N}_d$ denotes the image of $F_d|C_d$ in $\ol{M}_d$, and $\ol{K}_d$ is the kernel of the resulting map $ F_d | C_d \lra  \ol{N}_d$. 

Observe  now that 
\[  L_d | C_d \, \lin \, \big( C_d + B \big)|C_d \, \lin \, \big( K_X + C_d + Q)|C_d \]
for some very ample divisor $Q$ on $X$. In particular, $\deg (L_d|C_d) > 2g(C_d) + 1$, and hence $\ol{M}_d$ is stable on $C_d$ thanks to \cite{EL1}. On the other hand, it follows from  the bottom row of \eqref{Basic.Diagram} that $\mu(M_d|C_d) > \mu(\ol{M}_d)$, and hence
\begin{equation}
 \mu \big( F_d|C_d \big)  \ \ge \ \mu \big( M_d|C_d) \ > \ \mu\big(\ol{M}_d\big).
 \end{equation}
Therefore $F_d | C_d$ cannot be a subsheaf of $\ol{M}_d$, and hence $\ol{K}_d \ne 0$.

The following two lemmas constitute the heart of the proof. The first asserts that the destabilizing subsheaf $F_d$ must have large rank.
\begin{lemma} \label{F.Large.Rank.Lemma}
One has 
\[ \textnormal{rank} (F_d ) \ \ge \ h^0\big( (L_d-B) \otimes \frakm_x\big) = \ h^0\big(L_d - B \big) \, - \, 1. \]
\end{lemma}
\noi The second lemma shows that if $d$ is sufficiently large, then  the vertical inclusion on the left of \eqref{Basic.Diagram} is the identity. 
\begin{lemma}\label{K.Trivial.Lemma}
If $d \gg 0$, then $\ol{K}_d = H^0(B)_{C_d}$.
\end{lemma}

Granting these assertions for now, we give the
\begin{proof}[Proof of Theorem A] We need to show that if $d \gg 0$ then the picture introduced above  cannot occur. To this end, we consider the fibres at the fixed point $x \in X$ of the vector bundles appearing in the left hand square of \eqref{Basic.Diagram}. Recalling that the fibre $M_d(x)$ of  $M_d$ at $x$ is canonically identified with 
$ H^0(X, L_d \otimes \frakm_x)$, these take the form
\begin{equation}\label{Diagram.of.Fibres}
\begin{gathered}
\xymatrix{
& F_d(x)\ar@{^{(}->}[d]  \\
 H^0(X, B)  \ar@{^{(}->}[r]^-{ C_d}& H^0(X, L_d \otimes \frakm_x).
}
\end{gathered}
\end{equation}
Here the bottom map is the natural inclusion determined by a local equation for $C_d \in \linser{(L_d -B)\otimes \frakm_x}$. It follows from Lemma \ref{K.Trivial.Lemma} that $H^0(X,B)$ maps into the subspace 
\[ F_d(x) \ \subsetneqq \ H^0(X, L_d \otimes \frakm_x).\]
So for the required contradiction, it is enough to show that as $C_d$ varies over an open subset of $ \linser{(L_d -B)\otimes \frakm_x}$, the images of the corresponding embeddings of $H^0(X,B)$ span all of $H^0(X, L_d \otimes \frakm_x)$. But
this follows from the surjectivity of the map \eqref{Mult.C.Surjective}.
\end{proof}

\begin{proof}[Proof of Lemma \ref{F.Large.Rank.Lemma}]
We continue to work with the diagram \eqref{Diagram.of.Fibres}, and we write $\PPP( W )$ for the projective space of one-dimensional subspaces of a vector space $W$. Multiplication of sections gives rise to a finite morphism:
\[
\mu_d : \PPP\big(H^0(X, B)\big) \times \PPP \big( H^0(X, (L_d -B )\otimes \frakm_x) \big) \lra \PPP\big( H^0(X, L_d \otimes \frakm_x) \big). 
\]
Set 
\[ Z \ =\  \mu_d^{-1} \left( \PPP\big( F(x) \big) \right). \] 
Then 
\[ \dim \PPP\big(F(x)\big) \ \ge \ \dim Z \tag{*} \] thanks to the finiteness of $\mu_d$. On the other hand, for general $C_d \in \linser{(L_d - B) \otimes \frakm_x}$, the image of the corresponding inclusion 
\[  H^0(X, B) \ \subseteq \ H^0(X, L_d \otimes \frakm_x) \]
must meet the subspace $F(x) \subseteq H^0(X, L_d \otimes \frakm_x)$ non-trivially: indeed, this follows from \eqref{Basic.Diagram} and the fact that  $\ol{K_d}(x) \ne 0$. But this means that the projection
\[  \textnormal{pr}_2 : Z \lra \PPP\big( H^0(X, (L_d-B) \otimes \frakm_x)) \big) \  \tag{**}
\]
is dominant. The Lemma follows upon combining (*) and (**). 
\end{proof}

\begin{proof}[Proof of Lemma \ref{K.Trivial.Lemma}] 
Since $M_{d}/F_{d}$ is locally free along $C_{d},$ it follows from \eqref{Basic.Diagram} that $\overline{K}_{d}$ is a saturated subsheaf of $H^{0}(B)_{C_d},$ so it suffices to show that $\rk \ol{K}_d = h^0(B).$  The argument is numerical. First, note from \eqref{Basic.Diagram} and the stability of $\ol{M}_d$ that
\begin{equation} \label{Big.Eqn.1}
\begin{aligned}
\mu\big( F_d | C_d \big) \   = \ \frac{ \deg \ol{K}_d + \deg \ol{N}_d}{\rk F_d} \   & \le \ \frac{\deg \ol{N}_d}{\rk F_d}\\  & = \ \mu\big( \ol{N}_d\big) \cdot \left( \frac{ \rk \ol{N}_d}{\rk F_d} \right)\\ & <\ \mu\big(\ol{M}_d\big)\cdot \left( 1 - \frac{\rk \ol{K}_d}{\rk F_d} \right).  
\end{aligned}   \end{equation}
Now $\deg(M_d | C_d) = \deg(\ol{M}_d)$, and since
\[
\mu\big( M_d | C_d \big) \ \le \ \mu \big( F_d | C_d \big), \]
equation \eqref{Big.Eqn.1} yields:
    \[
\frac{ \deg (M_d | C_d)}{\rk \ol{M}_d + h^0(B)} \ < \  \frac{ \deg ( M_d|C_d)}{\rk \ol{M_d}}\cdot  \left( 1 - \frac{ \rk \ol{K}_d}{\rk{F_d}} \right).
\]
Observing  that  $\deg (M_d|C_d) < 0$, this is equivalent to the inequality\[
\frac{	1}{\rk \ol{M}_d + h^0(B)} \ > \ \frac{1}{\rk \ol{M}_d} \cdot \left( 1 - \frac{ \rk \ol{K}_d}{\rk{F_d}} \right),
\]
i.e.:
\[ \frac{{\rk \ol{M}_d}}{\rk \ol{M}_d + h^0(B)} \ > \  \left( 1 - \frac{ \rk \ol{K}_d}{\rk{F_d}} \right).\]
Thus
\[ \frac{\rk \ol{K}_d }{\rk F_d} \ > \ 1 -  \frac{{\rk \ol{M}_d}}{\rk \ol{M}_d + h^0(B)}\ = \ \frac{h^0(B)}{\rk M_d}, 
\]
and hence  
\[  \rk \ol{K}_d \ > \ h^0(B) \cdot \left ( \frac{ \rk F_d}{\rk M_d} \right). \tag{*}
\]
But by the previous lemma,  $\rk F_d \ge h^0(X, L_d -B) -1$. Furthermore,  $B$ is independent of $d$, and $\rk M_d = h^0(X, L_d) -1$. Thus as $d$ grows, the fraction on the right in (*) becomes arbitrarily close to $1$.\footnote{In fact, 
$\big(h^0(L_d) - h^0(L_d - B)\big) = O(d)$, whereas $h^0(L_d)$ grows quadratically in $d$.} It follows  that
\[  \rk \ol K_d \ > \ h^0(B) - 1 \]
provided that $d \gg 0$, and hence $\rk \ol{K}_d = h^0(B)$, as required.
\end{proof}

%\section{Coand\v{a}'s Argument}
\section{Complements}
In this section we first of all prove Proposition \ref{PropB.Intro} by adapting the method of proof of Theorem 1.1 in \cite{Coa}.  Then we propose some open problems. 

\subsection{Coand\v{a}'s Argument}
We begin by stating (without proof) two preliminary results on which the method rests; the first of these is a cohomological characterization of stability, and the second is a vanishing theorem of Green.

\begin{lemma}
\label{stabcrit}
Let $E$ be a vector bundle on $X.$  If for every $r$ with $0 < r < {\rm rk}(E)$ and for every line bundle $N$ on $X$ with $\mu_{L}(\Lambda^{r}E \otimes N) \leq 0$ one has $H^0(\Lambda^{r}E \otimes N)=0,$ then $E$ is ${L}-$stable. \hfill \qedsymbol
\end{lemma}

\begin{lemma}
\label{mlg}
\textnormal{(\cite[3.a.1]{Gr})} Let $N,N'$ be line bundles on $X$ and assume $N$ is very ample.  Then for $r \geq h^{0}(N'),$ we have $H^{0}(\Lambda^{r}M_{N} \otimes N')=0.$ \hfill \qedsymbol
\end{lemma}
\begin{proof}[Proof of Proposition \ref{PropB.Intro}]  Let $X$ be a smooth projective variety of dimension $n \geq 2$ for which ${\rm Pic}(X) \cong \mathbb{Z} \cdot [A]$ for an ample line bundle $A.$  Consider the function $q: \mathbb{N} \rightarrow \mathbb{Q}$ defined by $q(t)=\frac{h^{0}(tA)-1}{t}.$  Since $q(t)=\frac{A^{n}}{n!}t^{n-1}+O(t^{n-2})$ for $t >> 0$ by asymptotic Riemann-Roch, there exists a positive integer $d_0$ satisfying the following properties:
\begin{itemize}
\item[(1)]{For all integers $a$ satisfying $1 \leq a \leq d_{0}-1,$ we have $q(a) < q(d_{0}).$}
\vskip 5pt
\item[(2)]{For all integers $d \geq d_0,$ we have that $dA$ is very ample and $q(d) < q(d+1).$}
\end{itemize} 
An immediate consequence is that $q(a) < q(d)$ whenever $d \geq d_{0}$ and $1 \leq a \leq d-1.$  For the rest of the proof we fix an integer $d \geq d_0.$

Recalling that $\Pic(X) = \ZZ \cdot [A]$ by assumption, it suffices by Lemmas \ref{stabcrit} and  \ref{mlg} to show that given integers $a$ and $0 < r < h^{0}(dA)-1$, one has the implication
\begin{equation}
\mu_{A}(\Lambda^{r}M_{d} \otimes \mathcal{O}_{X}(aA)) \, \leq \, 0 \ \Longrightarrow \ r\,  \geq \, h^{0}(aA) ,
\end{equation}
where as before $M_d = M_{dA}$. This is automatic for $a \leq 0,$ so we assume $a \geq 1$ throughout.  We have that
\begin{equation}
{\mu}_{A}(\Lambda^{r}M_{d} \otimes \mathcal{O}_{X}(aA))\ =\ r\cdot {\mu}_{A}(M_{d})+a\cdot(A^{n}) \ =\ (A^{n}) \cdot \biggl(a-\frac{dr}{h^{0}(dA)-1}\biggr)
\end{equation}
Our assumption that $\mu_{A}(\Lambda^{r}M_{d} \otimes \mathcal{O}_{X}(aA)) \leq 0$ then implies that $a \leq \frac{d\cdot r}{h^{0}(dA)-1},$ or 
\begin{equation}
r \ \geq \ a\cdot \biggl(\frac{h^{0}(dA)-1}{d}\biggr).
\end{equation}
In particular, $a < d,$ so $1 \leq a \leq d-1.$  We will be done once we verify that
\begin{equation}
\label{reduction}
a\cdot \biggl(\frac{h^0(dA)-1}{d}\biggr)\  > \ h^{0}(aA)-1.
\end{equation}
for $1 \leq a \leq d-1.$  But (\ref{reduction}) is equivalent to $q(a) < q(d),$ so this follows from our assumption on $d.$  \end{proof}

\begin{remark}[Rigidity of $M_L$] Let $L$ be a very ample line bundle on a smooth complex projective variety $X$ of dimension $\ge 3$ with $H^1(X, \OO_X) = 0$. Then arguing as in the proof of \cite[Proposition 1]{Cam2}, one sees that $M_L$ is rigid, i.e. $\textnormal{Ext}^1(M_L, M_L) = 0$. Consequently, in the situation of Proposition \ref{PropB.Intro}, $M_d$ again represents an isolated point of the moduli space of bundles when $\dim_{\CC} X \ge 3$ and $d \gg 0$. 
\end{remark}
\subsection{Some Open Problems}

Recall that if $X$ is a smooth curve of genus $g$, then $M_L$ is stable as soon as $\deg L \ge 2g + 1$. This suggests
\begin{problem}
Find an effective version of Theorem A.
\end{problem}
\noi Presumably one would want to work with divisors of the sort $L = K_X + B + N$ with $B$ satisfying a suitable positivity hypothesis, and $N$ nef.

It is also interesting to ask whether $M_d$ satisfies some stronger stability properties:
\begin{problem} As before, let $L_d = dA + P$, and put $M_d = M_{L_d}$. Is $M_d$ slope stable with respect to \textit{any} polarization on $X$ when $d \gg 0$? In characteristic $p > 0$, is it strongly stable?
\end{problem}

Finally, we conjecture that our main result extends to all dimensions.
\begin{conjecture}
Let $X$ be a smooth projective variety of dimension $n$, and define $M_d$ as above. Then $M_d$ is $A$-stable for every $d \gg 0$.
\end{conjecture}

\end{document}